\documentclass[12pt]{amsart}
\usepackage{amsfonts,amsmath,amsthm,amssymb}

\newtheorem{theorem}{Theorem}
\newtheorem{corollary}{Corollary}
\newtheorem{proposition}{Proposition}
\newtheorem{lemma}{Lemma}

\newtheorem{remark}{Remark}

\newtheorem{question}{Question}

\title[On the endomorphism rings]{On the endomorphism rings of abelian groups and their Jacobson radical}
\begin{document}
\author{V.~Bovdi, A.~Grishkov,  M.~Ursul}
\dedicatory{Dedicated to the memory of Professor J\'anos ~Szendrei}

\thanks{The research was supported by FAPESP (Brazil) Process N: 2014/18318-7,
CNPq (Brazil) and RFFI (Russia) 13-01-00239a.}

\address{
\texttt{VICTOR BOVDI}
\newline
Department of Math. Sciences, UAE University, Al-Ain,
\newline United Arab Emirates} \email{vbovdi@gmail.com}

\address{
\texttt{ALEXANDER GRISHKOV}
\newline
IME, USP, Rua do Matao, 1010 -- Citade Universit\`{a}ria,
\newline
CEP 05508-090, Sao Paulo, Brazil} \email{shuragri@gmail.com}

\address{
\texttt{MIHAIL URSUL},
\newline
Department of Mathematics and Computer Science,
\newline
University  of Technology, Lae, Papua New Guinea}
\email{mihail.ursul@gmail.com}

\maketitle

\begin{abstract}
We give a  characterization of those abelian groups which are direct sums of cyclic groups and the Jacobson radical of
their endomorphism rings are closed. A complete characterization of $p$-groups $A$ for which $(EndA,\mathcal T_L)$ is locally compact,
where $\mathcal T_L$  is the Liebert topology on $EndA$, is given. We prove that if $A$ is a countable elementary $p$-group
then $EndA$ has a non-admissible ring topology. To every functorial topology on $A$ a right bounded ring topology on
$EndA$ is attached. By using  this topology we construct on $EndA$  a
non-metrizable and non-admissibe ring topology on $EndA$ for  elementary countable $p$-groups $A$.
\end{abstract}

\section{Introduction}

The question of closedness of the Jacobson radical of a topological ring have attracted the attention of researchers since  the beginning
of the theory of topological rings (see, e.g. \cite{Kaplansky_1}, \cite {Yood_1} and  \cite {Abel}).  Kaplansky  proved
\cite{Kaplansky_1} that the Jacobson radical of a compact ring is closed. Later he extended this result to locally compact
rings \cite{Kaplansky_2}.


 The ring of endomorphisms of an abelian group   considered as a topological ring with the finite topology
 is an interesting construction that connects the theory of rings, the theory of abelian groups and the theory of topological
rings.

 It is
tempting to use the Jacobson radical for  the study of rings of endomorphisms. But there is a  well-known example  in Bourbaki for
a free module $M_R$ where  $\frak{J}(End(M_R))$ is not closed (\cite{Bourbaki}, Chapter III, \S\, 6, p. 110).
The proof essentially uses
 the properties of $R$.
 The first examples of abelian groups $A$ with non-closed $\frak{J}(EndA)$ were constructed in
\cite{Ursul_Juras} and \cite{Budanov}.

We construct as in \cite{Ursul_Juras} examples such that   $\frak{J}(EndA)$ is not closed for certain groups $A$ that  are  direct sums of cyclic groups.

The properties of these examples raised the problem  of the characterization of abelian groups which are direct sums of cyclic groups and the
Jacobson radical of their endomorphism rings are closed. It was known that the Jacobson radical of a ring $EndA$ for a divisible group $A$ is closed, we propose a generalization of this result. Namely, we prove in Section 2 that the Jacobson radical $\frak{J}(End(M_R))$
for a quasi-injective module is closed.

One important class of abelian groups is the class of  direct sums of cyclic groups. We give in Section 3 a criterion  which
ensures that the
Jacobson radical of $EndA$ is closed when $A$ is a direct sum of cyclic groups.

One interesting topology on $EndA$  for a $p$-group $A$ is the Liebert topology $\mathcal T_L$. The notion of admissible
topologies for function spaces was introduced by \cite{Arens_Dugundji} and was extended in \cite{Abrudan} to topologies
on rings $EndA$. The importance of admissible topologies is based on the fact that if $A$ is an abelian group endowed with
the discrete topology and $\mathcal T$ a ring topology, then $A$ is a topological left $(EndA,\mathcal T)$-module if and
only if $\mathcal T$ is admissible.

The Jacobson radical of $(EndA,\mathcal T_L)$ is always closed. We prove that the Liebert topology is admissibile
(Proposition 4). We give a complete characterization of $p$-groups $A$ for which $(EndA,\mathcal T_L)$ is locally
compact (Theorem 4).
On the other hand, for every $p$-elementary countable group $A$ we construct  a non-admissible
metrizable topology (Theorem 5).

We study in Section 5 classes of ring topologies on $EndA$  associated with functorial topologies.   For
every $p$-elementary countable group $A$  we costruct  a countable elementary $p$-group such that  on $EndA$ there exists a non-discrete
admissible topology (Theorem 7). Up to the authors knowledge, no other examples for  a non-metrizable ring topology on $EndA$
for a countable $p$-elementary group have been found yet. We don't know if there exists an infinite ring on which every ring topology is
metrizable.

We give examples for the cases when a ring topology associated with a functorial topology  is discrete or non-discrete (Theorem 8).

\section{Notation}
The symbol $\omega$ stands for the set of all natural numbers including zero.  The set of all prime natural numbers
is denoted by $\mathbb{P}$. All rings are assumed to be associative with identity. We denote the Jacobson radical of
a ring $R$ by $\mathfrak{J}(R)$. All modules are unitary right modules. By $\overline{A}$ we denote the closure of
the subset $A$ of a topological space.

An abelian group $A$ is called bounded if there exists a non-zero natural number $n$ such that $na=0$ for all $a\in A$.
We denote the cyclic group of finite order $n$ with additive operation by $\mathbb{Z}(n)$.

We freely use facts about summable families in topological abelian groups (see \cite{Bourbaki}, Chapter III, \S 5):
If $A$ is an abelian group and $K$ is a finite subset of $A$, then put
\[
T(K)=\{\alpha\in EndA\mid  \alpha K=0\}.
\]
The finite topology on $EndA$ is given by the family $\{T(K)\}$,  where $K$ runs over all finite subsets of $A$ as
a fundamental system of neighborhoods of zero.

A family $\{x_i\}_{i\in I}$ of elements of $EndA$ is called summable (see \cite{Bourbaki},  p. 80) if there exists
$x\in EndA$ such that for any finite subset $K$ of $A$ there exists a finite subset $I_0\subset I$ such that
$\Sigma_{i\in I_1}x_i\in x+T(K)$ for any finite subset $I_1$ containing $I_0$.

It is well known (see \cite{Bourbaki}, Theorem 1, p. 82) that a family $\{x_i\}_{i\in I}$ of $EndA$ is summable in
$(EndA,\mathcal{T}_{fin})$  if and only if for each $a\in A$ there exists a finite subset $I_0$ of $I$ such that
$x_i(a)=0$ for all $i\notin I_0$.

All of the  necessary notions from the theory of abelian groups, ring   theory and the theory of topological groups can be
found in \cite{Fuchs}, \cite{Lambek} and  \cite{Bourbaki, Ursul_book}, respectively.

\section{Endomorphism rings of quasi-injective modules}

The question that when the Jacobson radical of a topological ring is closed has drawn the attention of researchers since
the beginning of the theory of topological rings (see for example \cite{Kaplansky_1}, \cite {Kaplansky_2}, \cite{Yood_1}
and  \cite{Yood_2}). In \cite{Kaplansky_1}  Kaplansky has proved  that the Jacobson radical of a compact ring is closed.
In  \cite{Kaplansky_2} he extended this result to locally compact rings.

It is interesting to connect  the theory of abelian groups and the theory of topological rings.
 The ring
of endomorphisms of an abelian group considered as a topological ring with the finite topology that will be denoted by $T_{fin}$.
It is tempting to use
the Jacobson radical to study the rings of endomorphisms. Bourbaki's example of a free module $M_R$ over a ring $R$
such that the Jacobson radical $\mathfrak{J}(End(M_R))$  of the ring of endomorphisms $End(M_R)$ is not closed
(see \cite{Bourbaki}, Chapter III, \S 6, p.110) is well known. The proof   uses essentially the  properties of the ring $R$.

A right module $M_R$ over a ring $R$ is called a quasi-injective module (see \cite{Faith_Utumi} or
\cite{Lambek}, Chapter 4, p.104),
if every partial endomorphism of $M_R$ can be extended to a full endomorphism.

\begin{theorem}\label{T:1}
The Jacobson radical $\frak{J}(End(M_R))$ of the ring of endomorphisms of a quasi-injective
module $M_R$ is closed with respect to the finite topology.
\end{theorem}
\begin{proof}[Proof of Theorem \ref{T:1}]
By \cite{Lambek}, \S 4.4, $\frak{J}(End(M_R))$ consists of  endomorphisms with large kernels.

Assume that there exists
$\alpha\in cl(\frak{J}(End(M_R)))\setminus \frak{J}(End(M_R))$. Since $\mathfrak{Ker}\alpha$ is not large there exists $0\neq m\in M$
such that
\[
\mathfrak{Ker}\alpha\cap mR=0.
\]
Let $\gamma\in T(m)$ be  such that $\alpha+\gamma=\beta\in \frak{J}(End(M_R))$.
Then $\mathfrak{Ker}\beta\cap mR\neq 0$. Let $s\in R$ be such that $0\neq ms\in \mathfrak{Ker}\beta$. Then $\alpha(ms)=0$,
hence $0\neq ms\in\mathfrak{Ker}\alpha\cap mR$, a contradiction.
\end{proof}

The next theorem provides an example in the reverse  direction.

\begin{theorem}\label{T:2}
Let $M_R$ be a  right quasi-injective module over a ring $R$.
The Jacobson radical $\mathfrak{J}(End(M_R))$ of the ring  $End(M_R)$  is closed with respect to the finite topology.
\end{theorem}

\begin{proof}[Proof of Theorem \ref{T:2}] Put $\mathfrak{J}_M=\mathfrak{J}(End(M_R))$.
It is well known (see \cite{Lambek}, \S 4.4) that  $\mathfrak{J}_M$ consists of  endomorphisms
with large kernels. Assume that there exists $\alpha\in \overline{\mathfrak{J}_M}\setminus \mathfrak{J}_M$.
Since $\mathfrak{Ker}(\alpha)$ is not large,  there exists $0\neq m\in M$ such that
$\mathfrak{Ker}(\alpha)\cap mR=0$. Let $\gamma\in T(m)$ be such that $\alpha+\gamma=\beta\in \mathfrak{J}_M$.
Then $\mathfrak{Ker}(\beta)\cap mR\neq 0$. If $s\in R$ satisfies  $0\neq ms\in \mathfrak{Ker}(\beta)$,
then $\alpha(ms)=0$, hence $0\neq ms\in\mathfrak{Ker}(\alpha)\cap mR$, a contradiction.
\end{proof}

Note that this  theorem  can be obtained  from  Lemma 11 of \cite{Mohamed_Muller}, but  our proof is based on different arguments.

\section{The finite and the Liebert topologies on endomorphism rings of abelian groups}

Examples of abelian groups $A$ such that $\mathfrak{J}(EndA)$ is not closed were constructed independently
in \cite{Budanov} and \cite{Ursul_Juras}.
Moreover, in  \cite{Ursul_Juras} an example of an abelian  group  $A$ was presented as a direct sum of cyclic groups
such that  the Jacobson radical  $\mathfrak{J}(EndA)$ is not closed. These examples raise the following
question:
\begin{question}\label{Q:1}
When does  the ring of endomorphisms of an  abelian group, which is a direct sum of cyclic groups,  have closed
Jacobson radical?
\end{question}

The following result gives a complete answer to Question \ref{Q:1}.

\begin{theorem}\label{T:3}
Let $A$ be a direct sum of cyclic groups. Then $\mathfrak{J}(EndA)$ is closed if and only if  every primary
component of $A$ is bounded.
\end{theorem}

Recall that the set of prime natural numbers is  denoted by $\mathbb{P}$.

\begin{lemma}\label{L7:2}Let $p$ be a prime number and $A=\mathbb{Z}(p^m)\times \mathbb{Z}(p^n)$ where $m<n$. We assume that $\mathbb{Z}(p^m)$
is embedded in $\mathbb{Z}(p^n)$ and $\mathbb{Z}(p^m)=p^{n-m}\mathbb{Z}(p^n)$. Let $\alpha_{m,n}:(x,y)\mapsto (0,x)$. Then $\alpha_{m,n}\in \frak{J}(EndA)$.
\end{lemma}
\begin{proof}
It follows from definition of $\alpha_{m,n}$ that $\alpha_{m,n}(A)\subset p^{n-m}A$, hence
$\beta_1\alpha_{m,n}(A)\subset p^{n-m}A$ for each $\beta_1\in EndA$. By induction,
\[
(\beta_i\alpha_{m,n})\cdots (\beta_i\alpha_{m,n})\in p^{(n-m)i}A
\]
for every $i\in\mathbb N$. This implies that $(EndA)\alpha_{m,n}$ is a nilpotent ideal, hence $\alpha_{m,n}\in \frak{J}(EndA)$.\end{proof}

\begin{lemma}\label{L:1}
Let  $A=\mathbb{Z}(p^m)\oplus  \mathbb{Z}(p^n)$,  where $p\in \mathbb{P}$ and  $m<n$. Let
$\tau: \mathbb{Z}(p^m)\rightarrow p^{n-m}\mathbb{Z}(p^n)\subset \mathbb{Z}(p^n)$ be an embedding homomorphism. Then
\[
\alpha_{mn}:(x,y)\mapsto \big(0,\; \tau(x)\big) \in \mathfrak{J}(EndA).
\]
\end{lemma}
\begin{proof}
Since $\alpha_{mn}(A)\subset p^{n-m}A$ by the definition of $\alpha_{mn}$,  we obtain  that
$\beta\alpha_{mn}(A)\subset p^{n-m}A$ for each $\beta\in EndA$. By induction on $k\in \omega$, we get that
\[
\underbrace{(\beta\alpha_{mn})\cdots (\beta\alpha_{mn})}_{k}\in p^{(n-m)k}A.
\]
This implies that $(EndA)\alpha_{mn}$ is a nilpotent left ideal, hence
\[
\alpha_{mn}\in \mathfrak{J}(EndA).
\]
\end{proof}

\begin{lemma}\label{L:2}
Let $A=\oplus_{i\in\omega}\mathbb{Z}(p^{k_i})$,  where $p\in \mathbb{P}$ and   $k_i<k_{i+1}$ for  $i\in\omega$.
Then $\mathfrak{J}(EndA)$ is not closed.
\end{lemma}
\begin{proof}
Let $\beta_i=\alpha_{k_ik_{i+1}}$  be the endomorphism in  Lemma \ref{L:1}, where $i\in\omega$. We extend $\beta_i$
to an endomorphism $\gamma_i$ of $A$, setting  $\gamma_i=\beta_i$ on $B_i=\mathbb{Z}(p^{k_i})\oplus \mathbb{Z}(p^{k_{i+1}})$
and $\beta_i(\mathbb{Z}(p^{k_l}))=0$ for $l\neq i$ or $l\neq i+1$.

The projection of $A$ onto $B_i$ is denoted by $e_i$. Consider the embedding $\psi_i:End(B_i)\to EndA$, where
\[
 \alpha\mapsto\psi_i(\alpha)=\begin{cases}\alpha& \text{if}\quad \alpha\in B_i;\\ 0& \text{if}\quad \alpha\not \in B_i.
\end{cases}
\]

Thus $\psi_i(End(B_i))=e_i(EndA)e_i$. Furthermore,
\[
\begin{split}
\gamma_i&=\psi_i(\beta_i)\in \mathfrak{J}(e_i(EndA)e_i)\\
&=\mathfrak{J}(EndA)\cap e_i(EndA)e_i\subset \mathfrak{J}(EndA).
\end{split}
\]
We claim that the family $\{\gamma_i\}_{i\in\omega}$ is summable. For if $a\in A$, then there exists $n\in\omega$ such that
$a\notin B_i$ for $i>n$. Then $\gamma_i(a)=0$ for $i>n$, hence $\{\gamma_i\}_{i\in\omega}$  is summable.

Let $\gamma=\sum_{i\in\omega}\gamma_i$. We claim that $\beta=\Sigma_{i\in\omega}\beta_i$ is not right quasi-regular. Assume
on the contrary that there exists $\beta^\prime\in EndA$ such that $\beta+\beta^\prime+\beta\beta^\prime=0$. Let
$x=(\theta00\ldots)\in A$ and let $\beta^\prime(\theta00\ldots)=(x_0x_1\ldots)$, where $\theta$ is a generator
of $\mathbb{Z}(p^{k_0})$.  Then
\[
(0\theta0\ldots)+(x_0x_1x_2\ldots)+(0x_0x_1\ldots)=0,
\]
so $x_0=0$,\quad  $x_1=-\theta$,\quad \ldots, $x_n=(-1)^n\theta$,\quad \ldots $(n\geq 2)$, a contradiction.
The last property implies that the ideal $\mathfrak{J}(EndA)$ is not closed.
\end{proof}

\begin{lemma}\label{L:3}
If $A=\oplus_{\alpha\in\Omega}H_\alpha$, where $H_\alpha=\mathbb{Z}(p^n)$ for all $\alpha\in\Omega$, then
\[
\mathfrak{J}(EndA)=pEndA \quad \text{and}\quad \mathfrak{J}(EndA)^n=0.
\]
\end{lemma}

\begin{proof}
Since $(pEndA)^n=0$, we have  $pEndA\subset \mathfrak{J}(EndA)$. Denote  a generator of $\mathbb{Z}(p^n)_\alpha$
by $\theta_\alpha$, where $\alpha\in \Omega$.

Assume that there exists an element $j\in \mathfrak{J}(EndA)\setminus pEndA$. Then there exists $\beta\in\Omega$
such that $j(\theta_\beta)\notin pA$. Indeed, otherwise for each $\alpha\in\Omega$ there exists $a_\alpha\in A$
such that $j(\theta_\alpha)=pa_\alpha$. The mapping $q$ from $\{\theta_\alpha\}_{\alpha\in\Omega}$ in $A$ which
sends $\theta_\alpha$ to $a_\alpha$ has an extension to an endomorphism $q_1$ of $A$ and $j=pq_1$, a contradiction.

Let $j(\theta_\beta)=k_1\theta_{\alpha_1}+\cdots+k_m\theta_{\alpha_m}$. We can assume without loss of generality
that $(k_1,p)=1$. There exists a natural number $k$ such that $kk_1\theta{\alpha_1}=\theta_{\alpha_1}$. Then
$kj\in\mathfrak{J}(EndA)$ and
\[
kj(\theta_\alpha)=\theta_{\alpha_1}+\cdots+k_m\theta_{\alpha_m}.
\]
Clearly  we can assume that
$j(\theta_\beta)=\theta_{\alpha_1}+\cdots+k_m\theta_{\alpha_m}$.

Let  $j_1\in EndA$,  such that $j_1(\theta_{\alpha_1})=\theta_\beta$ and $j_1(\theta_\gamma)=0$ for $\gamma\neq\alpha_1$.
Define $j_2\in EndA$, such that $j_2(\theta_\beta)=\theta_\beta$ and $j_2(\theta_\gamma)=0$ for $\gamma\neq \beta$.  Then
\[
j_1jj_2(\theta_\beta)=j_1j(\theta_\beta)=j_1(\theta_{\alpha_1}+k_2\theta_{\alpha_2}+\cdots+k_m\theta_{\alpha_m})=
\theta_\beta
\]
and $j_1jj_2(\theta_\gamma)=0$ for $\gamma\neq \alpha$. Therefore $0\neq j_2=j_1jj_2\in\mathfrak{J}(EndA)$ is an idempotent,
a contradiction.
\end{proof}

\begin{lemma}\label{L:4}
If  $A$ is  a bounded abelian group,  then $\mathfrak{J}(EndA)$ is a nilpotent ideal.
\end{lemma}

\begin{proof}
We can assume without loss of generality that $A$ is a $p$-group. It follows from  Pr\"ufer's theorem (see \cite{Fuchs}, Theorem
17.3, p. 88) that there exists a decomposition $A=A_{n_1}\oplus\cdots\oplus A_{n_k}$ such that  $n_1>\cdots> n_k$ and $A_{n_i}$
is a direct sum of copies of $\mathbb{Z}(p^{n_i})$ for  $1\leq i\leq k$.

We prove our lemma  by induction on $k$. For $k=1$  the statement of the lemma follows from Lemma \ref{L:3}.

Assume that our lemma has been proved for $k-1$. Consider the abelian groups $B=A_{n_1}\oplus\cdots\oplus A_{n_{k-1}}$
and  $C=A_{n_k}$ and also the following matrix rings
\[
U=\left(
   \begin{smallmatrix}
     Hom(B,B) & Hom(C,B) \\
         Hom(B,C) & Hom(C,C) \\
   \end{smallmatrix}
 \right)
\quad
and
\quad
W_U=\left(
   \begin{smallmatrix}
   \mathfrak{ J}(EndB) & Hom(C,B) \\
         Hom(B,C) & \mathfrak{ J}(EndC) \\
   \end{smallmatrix}
 \right).
\]
Clearly, $A=B\oplus C$ and  $EndA\cong U$.

We claim that the Jacobson radical $\mathfrak{J}(U)$ is $W_U$.

First we prove the following two inclusions:
\begin{equation}\label{E:1}
\begin{split}
Hom(C,B)Hom(B,C)&\subset\mathfrak{ J}(EndB);\\
Hom(B,C)Hom(C,B)&\subset \mathfrak{J}(EndC).\\
\end{split}
\end{equation}
Indeed, let $\alpha\in Hom(C,B)$, $\beta\in Hom(B,C)$ and $b\in B$. This yields that  $\beta(b)\in C$,
hence $\alpha\beta(b)\in pB$. Therefore $\alpha\beta(B)\subset pB$  since $n_i<n_k$ for all $i<k$.

If $\gamma\in EndB$, then $\gamma\beta\alpha(B)\subset pB$. It follows that $(\gamma\beta\alpha)^n=0$.
This implies that $\alpha\beta\in \mathfrak{J}(EndB)$, hence the first equation of (\ref{E:1}) is proved.

Now, let $\alpha\in Hom(B,C)$, $\beta\in Hom(C,B)$ and $c\in C$. Clearly, then $\beta(c)\in pB$, so
$\alpha\beta(b)\in pC$. We prove in a similar way that
\[
\alpha\beta\in \mathfrak{J}(EndC).
\]
It is easy to see that the map
\[
\psi: U\rightarrow  \big( EndB/\mathfrak{J}(EndB)\big)\times \big(EndC/\mathfrak{J}(EndC)\big),
\]
defined by
\[
\left(
   \begin{smallmatrix}
   \alpha_{11} & \alpha_{12} \\
       \alpha_{21} &\alpha_{22} \\
   \end{smallmatrix}
 \right)\mapsto \big(\; \alpha_{11}+\mathfrak{J}(EndB), \quad \alpha_{22}+\mathfrak{J}(EndC)\;\big)
\]
is a ring homomorphism, by (\ref{E:1}). Moreover $\mathfrak{J}(U)\subset \mathfrak{Ker}(\psi)= W_U$.

It suffices to show that the ideal $W_U$ is nilpotent.

Obviously $ T_1=\left(\begin{smallmatrix}
\mathfrak{J}(EndB) & 0 \\
            Hom(B,C) & 0 \\
\end{smallmatrix}
\right)$
is a left ideal of the ring $U$ and    from the inductive assumption we get  that  it  is nilpotent
of class $2n$.  Similarly the same is true for the left ideal $T_2= \left(\begin{smallmatrix}
     0 & Hom(C,B) \\
         0 & \mathfrak{J}(EndC)\\
   \end{smallmatrix}
 \right)$. This yields that    $ U_W=T_1+T_2$ is a nilpotent ideal of the ring $U$.
\end{proof}

\bigskip
\begin{corollary}\label{C:2}
The Jacobson radical of the ring of endomorphisms of a bounded  abelian group is closed.
\end{corollary}

Note that this  corollary  can be obtained from  theorem 6.1  of \cite{Mohamed_Muller}, but  our proof is based on different arguments.

\begin{lemma}\label{L:5}
If $A$ is a free group, then   $\mathfrak{J}(EndA)=0$.
\end{lemma}

\begin{proof} Put $\mathfrak{J}=\mathfrak{J}(EndA)$.
Assume on the contrary that $\mathfrak{J}\neq 0$. Let  $\{a_\alpha\}_{\alpha\in\Omega}$ be
a basis of $A$ and let $e_\alpha\in EndA$ such that
\[
e_{\alpha}(a_\beta)=\delta_{\alpha\beta}a_\alpha, \qquad   (\beta\in\Omega)
\]
where $\delta_{\alpha\beta}$ is the Kronecker $\delta$ function.

Clearly  $e_\alpha(EndA)e_\alpha\cong \mathbb{Z}$ and $e_\alpha(\mathfrak{J})e_\alpha=0$.

Let $b\in A$ and    $g\in EndA$ such that  $g(a_\alpha) =b$ for a fix $\alpha\in\Omega$.  Thus
\[
e_\alpha \mathfrak{J} ge_\alpha(a)=e_\alpha \mathfrak{J}(b)=0,
\]
hence $e_\alpha \mathfrak{J}=0$. Let $0\neq j\in\mathfrak{J}$, $\beta\in EndA$
and $h\in EndA$, such that
$j(a_\beta)=k_1a_{\alpha_1}+\cdots$,
where  $k_1\neq 0$,  $h(a_{\alpha_1})=a_\alpha$ and $h(a_\gamma)=0$ for all $\gamma\neq\alpha_1$. Then
$e_\alpha hj(a_\beta)=k_1a_\alpha\neq 0$, a contradiction.\end{proof}

\begin{lemma}\label{L:6}
Let  $M={}_RM_S$ be  an $(R,S)$-bimodule, where $R$ and  $S$ are rings. Put
$U=\left(
\begin{smallmatrix}
   R\mathfrak{}  &  M \\
   0   & S \\
\end{smallmatrix}
\right)
$ and
$W_U=\left(\begin{smallmatrix}
    \mathfrak{ J}(R)  & M \\
    0  & \mathfrak{J}(S) \\
\end{smallmatrix} \right)$.
Then
$\mathfrak{J}(U)=W_U$.
\end{lemma}

\begin{proof}
Clearly $W_U$ is a quasiregular ideal of the matrix ring
$U$,
hence is contained in $\mathfrak{J}(U)$. The map
\[
\Gamma: \left(\begin{smallmatrix}
   r  &  m \\
   0   & s \\
\end{smallmatrix} \right)
\mapsto  \big(r+\mathfrak{J}(R),\;  s+\mathfrak{J}(S)\big)
\]
is a ring homomorphism  from $U$ to $(R/\mathfrak{J}(R))\times (S/\mathfrak{J}(S))$
with the kernel  $\mathfrak{Ker}(\Gamma)=W_U$. This implies that $W_U\subset \mathfrak{J}(U)$.
\end{proof}
\begin{lemma}\label{L:7}
Let $A=B\oplus C$, where $B$ is a free group and $C$ is a torsion group. Then
\begin{itemize}
\item[(i)] $EndA\cong_{top}
    \left(\begin{smallmatrix}
    EndA &  0 \\
             Hom(B,C)  &  EndC \\
   \end{smallmatrix} \right)$,
where $EndA$, $Hom(B,C)$, and $EndC$ are endowed with the finite topology;
\item[(ii)] $\mathfrak{J}
   \left(\begin{smallmatrix}
    EndB  &  0 \\
             Hom(B,C)   & EndC \\
   \end{smallmatrix} \right)=
   \left(\begin{smallmatrix}
    0  &  0 \\
             Hom(B,C)   &  \mathfrak{J}(EndC) \\
   \end{smallmatrix} \right)$;
\item[(iii)]  $\mathfrak{J}(EndA)$ is closed if and only if $\mathfrak{J}(EndC)$ is closed.
\end{itemize}
\end{lemma}

\begin{proof}
(i) Follows from  Lemma 1  in \cite{Ursul_Juras}.
(ii) Follows from Lemma \ref{L:6}.
(iii) Since $Hom(B,C)$ is always closed, the assertion follows from (ii).
\end{proof}

\begin{proof}[Proof of Theorem \ref{T:3}]
Let $A=B\oplus C$, where $B$ is a free group and $C$ is a torsion group. The proof
follows from Lemmas \ref{L:4}, \ref{L:5} and \ref{L:7}.
\end{proof}

\section{The Liebert topology on endomorphism rings of $p$-groups}

In the sequel  $A$ denotes an abelian $p$-group. Denote  $A_n =A[p^n]$ for each $n\in\mathbb{N}$ and
$T(A_n)=\{ \alpha\in End(A)\mid \alpha A_n=0\}$. Thus the family
$\{T(A_n)\}$
is a fundamental system of neighborhoods of zero of a Hausdorff ring topology $\mathcal{T}_L$ on $EndA$.
This topology is called the \emph{Liebert topology}.
We note that each $T(A_n)$ is an ideal of $EndA$.

We prove now some properties of the Liebert topology $\mathcal{T}_L$ on $EndA$.
\begin{proposition}\label{P:1}
The ring $(EndA,\mathcal{T}_L)$ is a complete 1st countable topological ring.
\end{proposition}
\begin{proof}
Indeed, $T(A_n)=\cap_{a\in A_n}T(a)$, hence $T(A_n)$ is closed in the finite topology.
Since End$A$ endowed with the finite topology is complete, it is complete with respect
to the Liebert topology.
\end{proof}

\begin{corollary}\label{C:2}
For any $p$-group $A$ the ring  $(EndA,\mathcal{T}_L)$ is the inverse limit of discrete rings $EndA/T(A_n)$.
\end{corollary}
\begin{proposition}\label{P:2}(see \cite{Pierce} or  \cite{Fuchs}, Vol. II,   p. 224,  Exercise 8)
For any $p$-group the Jacobson radical of $(EndA,\mathcal{T}_L)$ is closed.
\end{proposition}
\begin{proof}
Follows from Proposition \ref{P:1} and Corollary on  page 46 in \cite{Kaplansky_3}. \end{proof}

\begin{proposition}\label{P:3}
Let $A$ be an abelian $p$-group. Then $(EndA,\mathcal{T}_L)$ is discrete if and only if $A$ is a bounded group.
\end{proposition}
\begin{proof}
There exists $n\in \mathbb{N}$ such that $T(A_n)=0$. Obviously, $A$ is a reduced group. If $A$ is unbounded, then
$A$ contains a subgroup isomorphic to $\mathbb{Z}(p^k)$ as a direct summand, so  $A=\mathbb{Z}(p^k)\oplus B$, where
$k>n$. Set $\alpha\in EndA$, $\alpha\in T(B)$, $\alpha(x)=p^{k-n}\mathbb{Z}(p^k)$. Thus $\alpha\neq 0$ and
$\alpha\in T(A_n)$, a contradiction.\end{proof}

Recall (see \cite{Abrudan}), that   a ring topology $\mathcal{T}$ on a ring $EndA$
is called \emph{admissible} if $\mathcal{T}\ge \mathcal{T}_{fin}$.

\begin{proposition}\label{P:4}
The Liebert topology on $EndA$ is admissible.
\end{proposition}
\begin{proof}
We have to show that the Liebert topology $\mathcal{T}_L$ is stronger than the finite topology $\mathcal{T}_{fin}$.
Let  $K$ be a finite subset of $A$. There exists $n\in\mathbb{N}$ such that $p^nK=0$. Then $K\subset A_n$, hence
$T(K)\supset T(A_n)$.\end{proof}

The following question arises naturally.

\begin{question}
Classify abelian $p$-groups whose endomorphism rings are locally compact with respect to the Liebert topology.
\end{question}

\bigskip
\begin{theorem}\label{T:4}
Let $A$ be an abelian $p$-group. The following statements  are equivalent:
\begin{itemize}
\item[(i)] $(EndA,\mathcal{T}_L)$ is compact;

\item[(ii)] $(EndA,\mathcal{T}_L)$ is locally compact;

\item[(iii)] $(EndA,\mathcal{T}_{fin})$ is compact;

\item[(iv)] $A$ is a  direct sums of finitely many  cocyclic groups.

\end{itemize}
\end{theorem}

\begin{proof}[Proof of Theorem \ref{T:4}]

$\rm(ii)\Rightarrow \rm(iii)$: There exists $n\in\mathbb{N}$ such that $T(A_n)$ is compact in $(EndA,\mathcal{T}_L)$.
Since $\mathcal{T}_L$ is admissible by Proposition \ref{P:4},  $T(A_n)$ is compact in  $(EndA,\mathcal{T}_{fin})$.

We claim that the maximal divisible subgroup $D$ of $A$ is the direct sum of a finite number of copies of $\mathbb{Z}(p^\infty)$.
Assume on the contrary that $A=B\oplus C$, where $B=\mathbb{Z}(p^\infty)\oplus\cdots\oplus \mathbb{Z}(p^\infty)\oplus\cdots$.
Consider the shift homomorphism
\[
\alpha:B\to B, (x_1x_2x_3\ldots)\mapsto (0x_1x_2\cdots).
\]
 Extend $\alpha$ to an endomorphism
of $A$ setting $\alpha(C)=0$ and $\alpha_{\upharpoonright_B}=\alpha$. Then  $p^n\alpha^m\in T(A_n)$ for  $m\in\mathbb{N}$.
It is easy to see that  the elements $p^n\alpha^ma$  are different for all $m\in\mathbb{N}$,  where
$a=((\frac{1}{p^{n+1}}+\mathbb{Z})00\ldots)$  (here the group $\mathbb{Z}(p^\infty)$ is the subgroup of $\mathbb{R}/\mathbb{Z}$
generated by the subset $\{\frac{1}{p^n}+\mathbb{Z}\mid n\in\mathbb{N}\}$). Thus $T(A_n)$ will not be compact in the finite
topology, a contradiction.

Let $A=B\oplus C$,  where $C$ is a direct sum of a finite number of copies of $\mathbb{Z}(p^\infty)$ and $B$ is reduced.
We prove  that $B$ is finite. Assume by way of contradiction, that $B$ is infinite.
Let $B^\prime=\oplus_{i\ge 1}\Big(\bigoplus_{\mathfrak{m}_i}\mathbb{Z}(p^i)\Big)$ be a basic subgroup of $B$.
Obviously, $B^\prime$ is unbounded. Let $i_0\in\mathbb{N}$ such that $\mathfrak{m}_{i_0}\ge 1$ and $i_0\ge n+1$.
There exists a subgroup $K$ of $A$ such that $A=\mathbb{Z}(p^{i_0})\oplus K$.  Let $a$ be a generator of $\mathbb{Z}(p^{i_0})$.
If $i_1>i_0$ and $\mathfrak{m}_{i_1}\ge 1$, then there exists an embedding $q_{i_1}$ of $\mathbb{Z}(p^{i_0})$ in
$\mathbb{Z}(p^{i_1})\subset \bigoplus_{\mathfrak{m}_{i_1}}\mathbb{Z}(p^{i_1})$. Extend $q_{i_1}$ to an endomorphism
of $A$ in the obvious way and keep the same notation. Then $p^nq_{i_1}\in T(A_n)$ and the set
$\{p^nq_{i_1}(a)\mid i_1\ge 1, \;  \mathfrak{m}_{i_1}\ge 1\}$ is infinite, a contradiction.

$\rm(iii) \Rightarrow \rm(iv)$: Follows from \cite{Fuchs}, Proposition 107.4.

$\rm(iv)\Rightarrow \rm(i)$:  For every $n\in\mathbb N$ the subroup $A[p^n]$ is finite, hence $T(A_n)\in \mathcal T_{fin}$.
Thus $\mathcal T_L\leq\mathcal T_{fin}$ and by compactness of $\mathcal T_{fin}$ we obtain that $\mathcal T_L=\mathcal T_{fin}$.
\end{proof}

\begin{remark}\label{R:2} Let $R$ be a ring. If  $e$ is an  idempotent of  $R$, then
\[
Ann_l(e)=R(1-e)=\{x-ex\vert x\in R\}.
\]
\end{remark}

\begin{theorem}\label{T:5}
Let $A$ be  a countable $p$-elementary group. Then the ring $R=EndA$  has a nonadmissible topology.
\end{theorem}

\begin{proof} [Proof of Theorem \ref{T:5}]
Let $\mathcal{T}$ be the ring topology having a fundamental system of neighborhoods of zero $Ann_r(K)$,
where $K$ is a finite subset of $R$. Let us show  that $\mathcal{T}$ is not admissible. Note that the
family $\{Ann_r(e)\}$  is a fundamental system of neighborhoods of zero of $(R,\mathcal{T})$,
where $e$ runs over all idempotents of $I_\omega$, where
\[
I_\omega=\{f\in EndA\; \mid \; |im(f)|<\infty \}.
\]
Indeed, for any finite subset $K\subset R$,  $Kx=0$ if and only if  $RKx=0$.
Since $R$ is regular, $RK=Re$ for an idempotent $e\in R$. Obviously, $e \in K$.

Now we will present a suitable fundamental system of neighborhoods of zero of $(R,\mathcal{T}_{fin})$. Consider
a basis $\{v_n\vert n\in\omega\}$ over $\mathbb F_p$. If $e\in I_\omega$, then $Ann_l(e)\in \mathcal{T}_{fin}$.

Indeed, $T(eA)=Ann_l(e)$. If $K$ is a finite subset of $A$, then
\[
T(K)=T(\langle K\rangle)=Ann_l(e),
\]
where $e$ is an idempotent of $R$ such that
$e\upharpoonright_{\langle K\rangle} =id_{ \langle K\rangle\ }$ and $e(K^\prime)=0$ for some complement $K^\prime$ of $K$.

We claim that $\mathcal{T}$ is not admissible. Assume the contrary. Let $0\neq e\in I_\omega$.
Then there exists an idempotent $f\in I_\omega$ such that
\[
(1-f)R=Ann_l(f)\subset Ann_r(e).
\]
Thus $(1-e)Rf=0$. Since $R$ is a prime ring, either $1-f=0$ or $e=0$, a contradiction.\end{proof}

\section{Topologies  on endomorphism rings of abelian groups associated with functorial topologies}

Let $\mathcal{A}$ be the class of all abelian groups. Assume that, for every group $A \in\mathcal{A},$ there is  a
topology $t(A)$  under which $A$ is a topological group (not necessarily Hausdorff). We call $t = \{ t( A )\vert A \in \mathcal{A}\}$
a functorial topology (see \cite{Fuchs}, p.33) if every homomorphism in $\mathcal A$ is continuous. An important example of a functorial
topology is the Bohr topology on abelian groups. Recall that if $A$ is an abelian group, then the Bohr topology $\mathcal T_{Bohr}$ is
the largest precompact topology on $A$. It is well known that $(A,\mathcal T_{Bohr})$ is Hausdorff. For further information about Bohr
topologies on abelian groups see  \cite{van_Douwen}.

Recall that a ring topology $\mathcal T$ on a ring $R$ is called right bounded if $(R,\mathcal T)$ has a fundamental system of
neighborhoods of zero consisting of right ideals of the multiplicative semigroup of $R$. The definition of a right bounded
topological ring given in \cite{Kaplansky_1} is  somewhat different, but it is equivalent to our notation.

\begin{theorem}\label{T:6}
Let $(A,\mathcal{T})$ be a functorial topology on a group $A$ and $\mathfrak B$ the set of all neighborhoods of zero of $(A,\mathcal{T})$.
For  each $V\in\mathfrak B$  the set $\{\alpha\in EndA\vert \alpha(A)\subset V\}$ is denoted by $P(V)$.

Then the family $\{P(V)\}_{V\in\mathfrak B}$ defines a Hausdorff right bounded ring topology on $EndA$.
\end{theorem}

\begin{proof}[Proof of Theorem \ref{T:6}]
Each set $P(V)$ is a right ideal of the multiplicative semigroup of $EndA$. Indeed, $\alpha\beta(A)\subset\alpha(A)\subset V$
for each $\alpha\in P(V)$ and $\beta\in EndA$.

Furthermore, if $V_1,V_2\in \mathfrak B$ and $V_1\subset V_2$, then $P(V_1)\subset P(V_2)$. Therefore, if $V_1,V_2,V_3\in \mathfrak B$
and $V_3\subset V_1\cap V_2$, then $P(V_3)\subset P(V_1)\cap P(V_2)$.

If $V_1,V_2\in \mathfrak B$ and $V_2-V_2\subset  V_1$, then $P(V_2)-P(V_2)\subset P(V_1)$. If $\alpha\in\cap_{V\mathfrak B}P(V)$, then
$\alpha(A)\subset\cap _{V\in\mathfrak B}V=0$, hence $\alpha=0$.

It follows that  the family $\{P(V)\}_{V\in\mathfrak B}$ defines a Hausdorff right boun\-ded ring topology on $EndA$.\end{proof}

\begin{theorem}\label{T:7}
Let $A$ be an elementary countable $p$-group and  let $\mathcal{T}_{Bohr}$ be the Bohr topology on $A$. If $\mathcal U$ is the ring
topology on $EndA$ associated with $\mathcal{T}_{Bohr}$,  then their supremum $\mathcal{T}_{Bohr}\vee\mathcal{T}_{fin}$ is a
nonmetrizable admissible topology.
\end{theorem}

\begin{proof}[Proof of Theorem \ref{T:7}]
The family $T(K)$,  when $K$ runs over all finite  subgroups of $A$, forms a fundamental system
of neighborhoods of zero of $(EndA,\mathcal{T}_{fin})$. Indeed, $T(K)=T(\langle K\rangle)$ and
$\langle K\rangle$ is a finite subgroup of $A$ for every finite subset $K$ of $A$.

Note that $\mathcal{T}_{Bohr}$ is not admissible. Indeed, let $A=\langle a\rangle\oplus B$ and
let  $P(V)\subset T(a)$ for some cofinite subgroup $V$ of $A$, where $a\not=0$. There exists
$\alpha\in End(A)$ such that $\alpha (A)\subset V$ and $\alpha(a)\neq 0$, hence $\alpha\in P(V\setminus T(a))$, a contradiction.

Let us  prove that  $(A,\mathcal{T}_{Bohr}\vee\mathcal{T}_{fin})$ is nonmetrizable. Assume on the contrary that there exists
a fundamental system $\{P(V_i)\cap T(K_i)\}_{i\in\omega}$ of neighborhoods of zero,  where each $V_i$ is a cofinite subgroup
of $A$ and each $K_i$ is a finite subgroup of $A$. Consider for each $i\in\omega$ the decomposition
$A=V_i^\prime\oplus (V_i\cap K_i)\oplus K_i^\prime\oplus W_i$, where $V_i^\prime\oplus(V_i\cap K_i)=V_i$ and $K_i=(V_i\cap K_i)\oplus K_i^\prime$.
We note that each $V_i^\prime$ is infinite.

Construct by induction the sequence $\{a_i\in V_i^\prime \mid i\in \omega\}$  of  linearly independent system of vectors over $\mathbb F_p$.
Let $\lambda:A\to \langle a_0\rangle$, where  $a_k\mapsto a_0$ for all $k\in\omega$. Then  $W$ is an open subgroup of $(A,\mathcal{T})$.
By assumption there exists $i\in\omega$ such that $P(V_i)\cap T(K_i)\subset P(W)$. Let $\gamma\in P(V_i)\cap T(K_i)$,\;
$\gamma_{\upharpoonright_{V_i^\prime}}=id_{V_i^\prime}$ and
\[
\gamma[(V_i\cap K_i)\oplus K_i^\prime\oplus W_i]=0.
\]
Thus $\gamma\in T(K_i)$ and  $\gamma(A)\subset V_i^\prime\subset V_i$, hence $\gamma\in P(V_i)\cap T(K_i)$. It follows that $\gamma\in P(W)$,
hence $\gamma(a_i)=a_i\in W$ and so $0=\lambda(a_0)=a_0$, a contradiction.
\end{proof}

\bigskip

Now we can ask the following natural question.

\begin{question}
What conditions on the abelian group $A$ ensure  that the ring topology $\mathcal U$ on $EndA$ discrete or  metrizable?

\end{question}

A particular answer to this question  is given by the following result.

\begin{theorem}\label{T:8}
Let   $A$ be one of the following groups:
\begin{itemize}
\item[(i)] a torsion-free group of cardinality $\leq 2^{\aleph_0}$;
\item[(ii)] a free group of cardinality $> 2^{\aleph_0}$.
\end{itemize}
Then $(EndA,\mathcal U)$  is discrete in the case $(i)$ and  non-discrete in the case $(ii)$.
\end{theorem}

\begin{proof}[Proof of Theorem \ref{T:8}]
(i) Embed $A$ in the compact group $\mathbb R/\mathbb Z$.   Since every infinite subgroup of  $\mathbb R/\mathbb Z$ is dense, $A$ is  a dense subgroup of  $\mathbb R/\mathbb Z$.
Moreover   $\mathbb R/\mathbb Z$  has a neighborhood of zero without non-zero subgroups, so the group  $A$ contains a neighborhood $V$ of zero which does not contain a non-zero
subgroup. Then $P(V)=0$, hence $(EndA,\mathcal U)$ is discrete.

(ii) We claim that every neighborhood of zero of $(A,\mathcal{T}_{Bohr})$ contains a non-zero subgroup.  Assume the contrary. Let $V$ be
a neighborhood of zero of   $(A,\mathcal{T}_{Bohr})$  which does not contain any non-zero subgroup.   Set $V_0=V$ and construct by induction
closed symmetric neighborhoods of zero of   $(A,\mathcal{T}_{Bohr})$ such that $V_n+V_n\subset V_{n-1}$ for  $n\in\omega$.  The closure $\overline{V_n}$
of each $V_n$ in $\widehat{A}$, where $\widehat{A} $ is the completion of  $(A,\mathcal{T}_{Bohr})$, is a neighborhood of zero of $\widehat{A}$.
Then $K=\cap\overline{ V_n}$ is  a compact subgroup of   $\widehat{A}$  and  $\widehat{A}/K $ is metrizable.

Let $\phi: \widehat{A}\to  \widehat{A}/K$ be the canonical homomorphism. We have that $K\cap A\subset \overline{V}\cap A=V$, hence $K\cap A=0$.
It follows that $\widehat {A}/K$ contains an isomorphic copy of $A$, hence $\vert \widehat {A}/K\vert>2^{\aleph_0}$, a contradiction.
\end{proof}

\smallskip

The authors would like to express their  gratitude to Dr. Pace Nielsen  and
particularly to the referee, for valuable remarks.

\end{document}